\documentclass[12pt]{article}

\usepackage{amsmath,amssymb}
\usepackage{yfonts}
\usepackage{xcolor}

\definecolor{El}{rgb}{.3,.7,1}
\usepackage[pdftex,
            pdfauthor={Mamhood},
            pdftitle={Testing Glossaries},
            pdfsubject={Glossaries Package for LaTeX},
            pdfkeywords={LaTeX, Package, Glossary},
            pdfproducer={Latex with hyperref},
            pdfcreator={PDFLaTeX},
            pdfduplex=DuplexFlipLongEdge,
            linkcolor=blue,
            citecolor=purple,
            colorlinks=true]{hyperref}
            
            \newcounter{pulse}[section]

\numberwithin{pulse}{section}
\numberwithin{equation}{section}
\newtheorem{theorem}[pulse]{\bf Theorem}
\newtheorem{proposition}[pulse]{\bf Proposition}
\newtheorem{lemma}[pulse]{\bf Lemma}
\newtheorem{corollary}[pulse]{\bf Corollary}

\newtheorem{cor}[pulse]{\bf Corollary}

\newtheorem{dummy-eg}[pulse]{\bf Example}
\newtheorem{dummy-rem}[pulse]{\bf Remark}
\newtheorem{dummy-def}[pulse]{\bf Definition}
\newenvironment{eg}{\begin{dummy-eg}\upshape}{\end{dummy-eg}\ignorespacesafterend}
\newenvironment{rem}{\begin{dummy-rem}\upshape}{\end{dummy-rem}\ignorespacesafterend}

\newenvironment{dfn}{\begin{dummy-def}\upshape}{\end{dummy-def}\ignorespacesafterend}

\newenvironment{proof}{\noindent{\it Proof.}\/}{\hfill$\Box$\newline\ignorespacesafterend}
\newcommand{\supp}{\operatorname{supp}}
\newcommand{\conj}{\operatorname{Conj}}
\renewcommand{\ker}{\operatorname{Ker}}
\newcommand{\cS}{{\cal S}}
\newcommand{\cA}{{\cal A}}
\newcommand{\cB}{{\cal B}}
\newcommand{\oA}{{\Omega_{\cal A}}}
\newcommand{\ooA}{{\overline{\Omega}_{\cal A}}}
\newcommand{\norm}[1]{\|#1\|}
\newcommand{\oB}{{\Omega_{\cal B}}}
\newcommand{\Tg}{{T^{t}y^{t}}}
\newcommand{\oT}{{\overline{T}}}
\newcommand{\gA}[1]{{{\cal A}(G_{#1})}}

\newcommand{\cM}{{\cal M}}
\newcommand{\Lpa}{{\operatorname{Lip}_\alpha X}}
\newcommand{\lp}{{\operatorname{lip}_\alpha X}}

\newcommand{\coz}{{\operatorname{coz}}}
\newcommand{\Nat}{{\mathbb N}}

\date{21 October 2013}
\begin{document}

\title{{\huge Separating maps between  commutative Banach algebras }}
\author{{\large Mahmood Alaghmandan, Rasoul Nasr-Isfahani and Mehdi Nemati}}

\maketitle

%\tableofcontents
%\vfill\eject

%%%%%%%%%%%%%%%%%%%%%%%%%%%%%%%%%%%%%%%%%%%%%%%%%%%%%%%%%%%%%%%%%%%%%%%%%%%%%%%
%%%%%%%%%%%%%%%%%%%%%%%%%%%%%%%%%%%%%%%%%%%%%%%%%%%%%%%%%%%%%%%%%%%%%%%%%%%%%%%

\begin{abstract}
Let $\cA$ and $\cB$ be Banach algebras. A linear map $T:\cA \rightarrow \cB$ is called separating or disjointness preserving if $ab=0$ implies $Ta\;Tb = 0$ for all $a,b\in \cA$. In
this paper, we study a new class of  regular Tauberian algebras and prove that some well-known Banach algebras  in harmonic analysis  belong to this class. We show that a bijective separating map between these algebras turns out to be continuous and the maximal ideal spaces of underlying algebras are homeomorphic.
By imposing extra conditions on these algebras, we find a more thorough characterization of separating maps. The existence of a bijective separating map also leads to the existence of an algebraic isomorphism in some cases.
\end{abstract}

%%%%%%%%%%%%%%%%%%%%%%%%%%%%%%%%%%%%%%%%%%%%%%%%%%%%%%%%%%%%%%%%%%%%%%%%%%%%%%%
%%%%%%%%%%%%%%%%%%%%%%%%%%%%%%%%%%%%%%%%%%%%%%%%%%%%%%%%%%%%%%%%%%%%%%%%%%%%%%%

{\it Separating maps} (also considered under the name of {\it disjointness preserving
maps}) between general vector lattices were studied by several authors, for example \cite{ab, ar}. Separating maps were later considered in \cite{be} for spaces of continuous
functions defined on compact Hausdorff spaces. Subsequently, several results were gained on locally compact groups. For example, Font and Hernandez studied separating maps between the algebra of bounded continuous maps on locally compact groups in \cite{fh3,fh1}. In \cite{fh2}, they used the separating maps between Fourier algebras of two locally compact abelian groups to study the separating maps on group algebras of locally compact abelian groups.
%Later on, the results gained in \cite{fh2} led to a study on separating maps between Fourier algebras of two arbitrary locally compact groups (not necessarily abelian) in \cite{fh4}.
Recently, Alaminos, Bresar, Extremera, and Villena in  \cite{al} attained a characterization for the continuous separating maps between group algebras of locally compact (not necessary abelian) groups. Their study also led to a characterization of continuous separating maps between $C^*$-algebras.
Preparing the very last drafts of this manuscript, we saw \cite{la-sep}. In this recent manuscript, Lau and Wong improve the previous results about separating maps between Fourier algebras by adding  some kind of orthogonality property to separating maps. 

In this paper, we study the separating maps between some classes of semisimple regular commutative Banach algebras. In Section~\ref{s:prelimenaries}, we start with establishing our notation and recalling some introductory definitions and facts.

 In Section~\ref{s:M-algebras}, we define a condition for semisimple regular commutative Banach algebras, called condition $(M)$. We prove that some classes of well-known Banach algebras are satisfying condition $(M)$, namely, some classes of Figa-Talamanca Herz Lebesgue algebras, Mirkil algebras, and center of group algebras for compact groups.

In \cite{fo}, the separating maps between regular commutative Banach algebras  satisfying Ditkin's condition were studied.   A brief look at \cite{fo} verifies that in that paper the complete potential of these Banach algebras has not been used.
In Section~\ref{s:automatic-continuity}, we prove that the main results of \cite{fo} are still valid if Ditkin's condition is replaced with  condition $(M)$. Here, we may notice that there are some algebras which are  not satisfying Ditkin's condition while they satisfy condition $(M)$.

%In \cite{fh3}, by imposing some conditions to the separating maps from  $C_0(G_1)$ onto $C_0(G_2)$, the $C^*$-algebra of  all continuous functions vanishing at infinity on locally compact groups $G_1$ and $G_2$ respectively, a group isomorphism between $G_1$ and $G_2$ has been shown. In Section~\ref{s:algebraic-characterization-of-groups}, we generalize this result for a specific family of function algebras over locally compact groups which we call  convolution function algebras.  It is shown that bijective separating maps between these algebras satisfying one extra condition implies that the underlying groups are isomorphic.

Eventually, in Section~\ref{s:BSE-condition}, we study separating maps between  BSE-algebras satisfying condition $(M)$.  Here, we  add a new family of BSE-algebras to the known examples  in \cite{ka3, ta}, which is the center of group algebras for  compact groups. Subsequently,
%we adopt the structure of \cite{fo0} for Tauberian BSE-algebras satisfying condition $(M)$.  Imposing
imposing the existence of an approximate identity norm bounded by $1$, it is proven that the existence of a bijective separating map also leads to this fact that underlying algebras are algebraically isomorphic. A similar result has been proven in \cite{fo0} for BSE-algebras satisfying Ditkin's condition.

\begin{section}{Preliminaries}\label{s:prelimenaries}

Let $\cal A$ be a commutative Banach algebra. We denote by $\oA$ the {\it maximal ideal space} of ${\cal A}$ which is also called the {\it spectrum} or {\it character space} of ${\cal A}$ equipped with the {\it Gelfand topology}.
A commutative Banach algebra $\cal A$ is called {\it regular}, if for every $x\in\oA$  and every open neighbourhood $U$ of $x$ in the Gelfand topology, there exists an element $a\in {\cal A}$ such that $\widehat{a}(x)=1$ and $\widehat{a}$ is zero on $ \oA\setminus U$.

A commutative Banach algebra $\cA$ is called semisimple if $\cap_{x\in \oA} \ker x = \emptyset$.
Using the Gelfand representation theory, for a semisimple regular commutative  Banach algebra $\cA$, $\widehat{\cA}:=\{\widehat{a}:\; a\in\cA\}$ forms an algebra of continuous functions on its maximal ideal space vanishing at infinity, $\widehat{\cA}\subseteq C_0(\oA)$.  One may notice that if $\cA$ is a semisimple regular commutative Banach algebra, $\widehat{\cA}$ is dense in $C_0(\oA)$. Let us define for $a\in \cA$, $\coz(\widehat{a}):=\{x\in\oA:\ \widehat{a}(x)\neq 0\}$ and $\supp(\widehat{a})$ to be the closure of $\coz(\widehat{a})$.

We denote the set of all elements  $a\in \cA$ such that $\supp(\widehat{a})$  is compact by $\cA_c$. A semisimple commutative Banach algebra $\cA$ is called a {\it Tauberian algebra} when  $\cA_c$ is dense in $\cA$.

Let $\cA$ be a semisimple regular commutative Banach algebra.
By \cite[Corollary~4.2.10]{ka}, for every compact set $K\subseteq \oA$ and every open neighborhood $U$ of $K$, there exists some $a_{K,U}\in\cA$ such that
\begin{enumerate}
%\item{$\widehat{a}_{K,U}\big(\oA\big) \subseteq [0,1]$,}
\item{$\widehat{a}_{K,U}|_{K} \equiv 1$,}
\item{ $\supp(\widehat{a}_{K,U}) \subseteq U$.}
\end{enumerate}
%From now on, we use $a_{K,U}$ to denote an element of $\cA$ which satisfies (i), (ii), and (iii) for corresponded $K$ and $U$. To facilitate the writing, we use $a_{x,U}$ for $a_{\{x\},U}$ where $x\in\oA$.

\begin{lemma}\label{l:1-bdd-approximate-identity}
Let $\cA$ be a semisimple regular commutative Tauberian algebra. Then $\cA$ has an approximate identity $\norm{\cdot}_\cA$-bounded by some $D>0$ if and only if for each $\epsilon>0$ and every compact set $K\subseteq \oA$, there exists some $a_{K}\in \cA_c$ such that $\widehat{a}_{K}|_K \equiv 1$ and $\norm{a_{K}}_\cA \leq D + \epsilon$.
\end{lemma}

\begin{proof} Suppose that $(e_\alpha)_\alpha$ is a bounded approximate identity of $\cA$ such that $\norm{e_\alpha}_\cA \leq D$ for some $D>0$.
For each $K\subseteq \oB$, let $a\in \cA_c$ such that $a|_K\equiv 1$ and $I_K$ be the closed ideal $\{b\in \cA:\; \widehat{b}(K)=\{0\}\}$ of $\cA$. Therefore, for each $b\in \cA$, $ba -b \in I_K$. Considering the quotient norm of $\cA/I_K$, one gets
 \[
 \norm{a + I_K}=\lim_\alpha \norm{ ae_\alpha + I_K} = \lim_\alpha \norm{ e_\alpha + I_K}\leq D.
 \]
 So, there is some $b\in \cA_c \cap I_K$ such that $\norm{ a + b }_\cA < D+\epsilon$. Just note that $(a+b)|_K\equiv 1$ and $a+b \in \cA_c$.

 Conversely, for each $\epsilon>0$ and $K\subseteq \oA$ compact, let $\widehat{a}_{K,\oA}$ such that $\widehat{a}_{K,\oA}|_K \equiv 1$ and $\norm{a_{K,\oA}}_\cA \leq D (1 + \epsilon)$. Define $e_{K,\epsilon}:=(1+\epsilon)^{-1}a_{K,\oA}$. It is not hard to show that $(e_{K,\epsilon})_{K,\epsilon}$ forms an approximate identity of the Tauberian algebra $\cA$ which is $\norm{\cdot}_\cA$-bounded by $D$ where $\epsilon \rightarrow 0$ and $K \rightarrow \oA$.
\end{proof}

%\begin{cor}\label{c:W-bdd-approximate-identity}
%Let $\cA$ be a semisimple regular commutative Tauberian Banach algebra. Then $\cA$ has a $\norm{\cdot}_\cA$-bounded approximate identity  by some $D>0$ if and only if for each $\epsilon>0$ and every compact set $K\subseteq \oA$, there exists some $a_{K,\oA}\in \cA_c$ such that $\widehat{a}_{K,\oA}|_K \equiv 1$ and $\norm{a_{K,\oA}}_\cA \leq D + \epsilon$.
%\end{cor}

%\begin{proof}

%Fix $\epsilon>0$.
%Let $(e_\alpha)_\alpha$ be a $\norm{\cdot}_\cA$-bounded approximate identity of $\cA$ such that $\norm{e_\alpha}_{\cA}\leq D$ for all $\alpha$.
%Suppose that $\cB:=\overline{\cA_c}^{\norm{\cdot}_\cA}$. Note that $\cB$ is another semisimple regular commutative Banach algebra and $\oA=\oB$.
%Let us generate a net  $(e'_{\delta,\epsilon})_{\alpha, \epsilon/2>\delta>0}$ in $\cB$
% where  $\norm{e'_{\alpha,\delta} - e_\alpha}_1<\delta$ for each pair of $(\alpha,\delta)$.  We show that $(e'_{\alpha,\delta})_{\alpha, \epsilon/2 > \delta >0}$ is a  $\norm{\cdot}_\cA$-bounded approximate identity of $\cB$. Doing so, for each $b\in \cB$, note that
%\begin{eqnarray*}
%\norm{b e'_{\alpha,\delta} - b}_\cA  \leq \norm{b e'_{\alpha,\delta} - b e_\alpha}_{\cA}+\norm{b e_{\alpha} - b}_{\cA} \leq %\norm{b}_\cA  \norm{ e'_{\alpha,\delta}- e_\alpha}_{\cA} +  \norm{e_{\alpha}b - b}_{\cA} \rightarrow 0
%\end{eqnarray*}
%where $\alpha $ grows and $\delta \rightarrow 0$. Moreover, $\norm{ e'_{\alpha,\delta}}_\cA \leq D+\epsilon/2 $ for all %$\delta$ and $\alpha$.
%\end{proof}

From now on, we say $\cA$ has a $D$-bounded approximate identity if $\cA$ has an approximate identity whose norm is bounded by $D\geq 0$.

\begin{dfn}\label{d:D-regular}
We call a regular commutative Banach algebra $\cA$ to be {\it $D$-uniformly regular} for some $D>0$ if for each $x\in \oA$ and open neighborhood $U$ of $x$,
\[
\inf\{ \norm{a}:  \ \text{ $a\in\cA$,  $\widehat{a}(x)=1$, and $\supp\widehat{a} \subseteq U$}\} \leq D.
\]
\end{dfn}

Therefore, by Lemma~\ref{l:1-bdd-approximate-identity}, every regular commutative Banach algebra $\cA$ which has a $D$-bounded approximate identity is $D$-uniformly regular.
\begin{eg}\label{eg:non-amenable-no-condition-M}
For a   discrete group $G$ containing the free group on two generators, it is well
known that $G$ is not amenable. Let $A_2(G)$ be the Fourier algebra of $G$. For an infinite set $S$  of $G$, $M := \{\varphi \in A_2(G)^* :\ \supp(\varphi)\subseteq S\}$ is isometrically isomorphic to $\ell^2(S)$, \cite{le}. It follows that $M$ is the dual space of $A_2(G)/J_S$ where $J_S$ is the closed ideal that is the closure of $\{f\in  A_2(G)\cap c_c(G): \supp(f)\cap S=\emptyset\}$. Hence
$A_2(G)/J_S$ is isomorphic to $\ell^2(S)$. Consequently, there exists a constant $C > 0$ such that
\[
\norm{f}_{A_2(G)}^2 \geq \norm{f+J_S}_{A_2(G)/J_S} \geq C \sum_{x\in S} |f(x)|^2
\]
for all $f\in  A_2(G)$. This shows that the algebra $A_2(G)$ is not $D$-uniformly regular for all $D>\sqrt{C}$.
\end{eg}

In Subsection~\ref{ss:Lip-algebra}, we see some more regular commutative Banach algebras which are not   satisfying $D$-uniform regularity for some $D$.

\vskip1.5em

We say that a Banach algebra $({\cS_\cA},\|\cdot\|_{{\cS_\cA}})$ is an {\it abstract Segal algebra} of a Banach algebra
$({\cA},\|\cdot\|_{\cA})$ if \\
\begin{enumerate}
\item[(S1)]{ ${\cS_\cA}$ is a dense left ideal in ${\cA}$.}
\item[(S2)]{There exists $C>0$ such that $\|b\|_{\cA} \leq C
\|b\|_{\cS_\cA}$
for each $b\in {\cS_\cA}$.}
\item[(S3)]{There exists $M>0$ such that $\|ab\|_{\cS_\cA}\leq
M\|a\|_{\cA}\|b\|_{\cS_\cA}$ for all  $a\in \cA$ and $b \in {\cS_\cA}$.}
\end{enumerate}

\vskip1.5em

\begin{proposition}\label{p:regularity-of-asa}
Let ${\cal A}$ be a semisimple  regular commutative  Banach algebra, for some $D>0$, and let $\cS_\cA$ be an abstract Segal algebra of $\cal A$. Then $\cS_\cA$ is also a semisimple  regular algebra whose maximal ideal space is $\oA$. Moreover, $\cA_c\subseteq \cS_\cA$.
\end{proposition}

\begin{proof}
By \cite[Theorem~2.1]{bu}, $\Omega_{\cS_\cA}$ is homeomorphic to $\oA$ and ${\cS_\cA}$ is semisimple. The rest has been shown in \cite[Proposition~2.2]{ma}.
\end{proof}

\end{section}

\vskip2.0em
%%%%%%%%%%%%%%%%%%%%%%%%%%%%%%%%%%%%%%%%%%%%%%%%%%%%%%%%%%%%%%%%%%%%%%%%%%%%%%%%%%%%%%%
%%%%%%%%%%%%%%%%%%%%%%%%%%%%%%%%%%%%%%%%%%%%%%%%%%%%%%%%%%%%%%%%%%%%%%%%%%%%%%%%%%%%%%%

\begin{section}{Condition $(M)$, its properties and examples}\label{s:M-algebras}

\begin{dfn}\label{d:M-algebra}
Let $\cA$ be a commutative Banach algebra. We say that $\cA$ satisfies {\it condition $(M)$} if
it is regular and for each $x\in \oA$, $a\in \cA$ with $\widehat{a}(x)=0$ and  each neighborhood $U$  of $x$,
\[
\inf\{ \norm{ba}_\cA: \text{$b\in \cA$,  $\widehat{b}|_V\equiv 1$ for some neighborhood $V$ of $x$ and  $\supp(\widehat{b})\subseteq U$}\}=0.
\]
\end{dfn}

  \vskip2.0em

%\begin{rem}\label{r:always-a-can-be-chosen-compact-supported}
%In Defintion~\ref{d:M-algebra}, since for each $U$, we may find some $U'\subseteq U$ relatively compact, $a_{K,U}$ can be %chosen in $\cA_c$. Moreover, if $\cA$ is an $M$-algebra which satisfies  condition $(S)$, clearly $\cA$ is an $M_D$-algebra %for all $D>1$.
%\end{rem}

%The following corollary is a direct result of Lemma~\ref{l:1-bdd-approximate-identity}.

%\begin{cor}\label{c:D-Tauberian-&-D-bai}
%Let $\cA$ be an $M$-Tauberian algebra which has a $D$-bounded approximate identity. Then $\cA$ is an $M_D$-Tauberian algebra.
%\end{cor}

\begin{eg}\label{eg:C0(X)}
Let $X$ be a locally compact Hausdorff space. It is straightforward to check that $C_0(X)$, the $C^*$-algebra of all bounded continuous functions on $X$ vanishing at infinity. Therefore, any commutative $C^*$-algebra satisfies condition $(M)$.
\end{eg}

 \vskip1.5em

\begin{lemma}\label{l:discrete-SP}
Let $\cA$ be a regular commutative Banach algebra and $\oA$ equipped with the Gelfand topology is a discrete space. Then $\cA$ is a  Tauberian algebra that satisfies condition $(M)$.
\end{lemma}

To prove this lemma, note that since $\cA$ is a regular commutative Banach algebra and $\oA$ is discrete, $\cA$ automatically contains all point-mass functions.

For $\phi\in \oA$ the notion of {\it $\phi$-contractibility} of  Banach
algebras was recently introduced and studied by Hu, Monfared and
Traynor \cite{hmt}. In fact, $\mathcal{A}$ is called {\it
$\phi$-contractible} if there exists a (right) $\phi$-diagonal;
i.e., an element ${\bf m}$ in the projective tensor product
${\mathcal{A}}\widehat\otimes{\mathcal{A}}$ such that
$$
\phi(\pi({\bf m}))=1\quad\hbox{and}\quad a\cdot{\bf m}=\phi(a){\bf
M}
$$
for all $a\in{\mathcal{A}}$, where $\pi$ denotes the product
morphism from ${\mathcal{A}}\widehat\otimes{\mathcal{A}}$ into
${\mathcal{A}}$ given by $\pi (a\otimes b)=ab$ for all
$a,b\in\mathcal{A}$.
\begin{corollary}\label{contra}
Let $\cA$ be a regular commutative  Banach algebra which is  $\phi$-contractible for all $\phi\in \oA$. Then $\cA$ is a  Tauberian algebra that satisfies condition $(M)$.
\end{corollary}
\begin{proof}
This follows from the fact that $\oA$ is
discrete with respect to the  Gelfand topology, when
 $\cA$ is  $\phi$-contractible  for all
$\phi\in \oA$; see \cite{dns}, Proposition 2.3.
\end{proof}

The concept of pseudo-contractibility for Banach algebras was introduced and
investigated by Ghahramani and Zhang \cite{7} according to the
existence of a central approximate diagonal; i.e., a net $({\bf
m}_\alpha)$ in ${\mathcal{A}}\widehat\otimes{\mathcal{A}}$ such
that

$$ \|a\pi({\bf
m}_\alpha)-a\|\rightarrow 0 \quad\hbox{and}\quad a\cdot{\bf
m}_\alpha={\bf m}_\alpha \cdot a
$$
for all $a\in{\mathcal{A}}$ and all $\alpha$.
\begin{corollary}
Let $\cA$ be a regular commutative  Banach algebra which is  pseudo-contractible. Then $\cA$ is a  Tauberian algebra that satisfies condition $(M)$.
\end{corollary}
\begin{proof}
This follows from  Corollary \ref{contra}, together with the fact that any pseudo-contractible 
Banach algebra is $\phi$-contractible; see \cite{ANN}, Theorem 1.1.
\end{proof}
%%%%%%%%%%%%%%%%5A remark by YC which is not correct anymore%%%%%%%%%%%%%%%%
%\begin{rem}
%Let $\cA$ be a commutative Banach algebra with an approximate identity (not necessarily bounded) and also every maximal ideal of $\cA$ has a bounded approximate identity (not necessarily bounded). We show that $\cA$ satisfies condition $(M)$. Let $x\in \oA$ and $a\in \cA$ such that $\widehat{a}(x)=0$. If $b\in \cA$ such that $\widehat{b}(x)=1$. For given $ 0<\epsilon<1$, one can find $v\in\cA$ and $u\in\ker(x)$ such that $\norm{vb-b}_\cA < 1/3$, $\norm{va-a}_\cA<\epsilon/3$, and $\norm{ua-a}_\cA < \epsilon/3$. Note that
%\[
%|\widehat{v}-1| = |\widehat{vb}(x)-\widehat{b}(x)| \leq \norm{vb-%b}_\cA < 1/3.
%\]
%Hence, $|\widehat{v}(x)|^{-1}\leq 3/2$. Put $c = |\widehat{v}(x)|^{-1} (v-u)$ and note that $\widehat{c}(x)=1$. Also,
%\[
%\norm{ca}_\cA=|\widehat{v}(x)|^{-1}\norm{va-ua}_\cA \leq \frac{3}{2} %(\norm{va-a}_\cA + \norm{a-ua}_\cA) < \epsilon.
%\]
%\end{rem}

 \begin{proposition}\label{p:inheritent}
 Suppose that $\cS_\cA$ is an abstract Segal algebra with respect to a Banach algebra $\cA$. If $\cS_\cA$ is a Tauberian algebra which satisfies condition $(M)$, then  so is $\cA$.
 \end{proposition}

 \begin{proof}
By Proposition~\ref{p:regularity-of-asa},  $\Omega_{\cS_\cA}=\oA$ and $\cA_c\subseteq \cS_\cA$. Since $\cS_c (=\cA_c)$ is $\norm{\cdot}_{\cS_\cA}$-dense in $\cS$, $\cA_c$ is $\norm{\cdot}_\cA$-dense in $\cA$, by (S2). The rest of the proof is straightforward.
 \end{proof}

 Note that in the proof of Proposition~\ref{p:inheritent}, we did not apply condition (S3) of the definition of abstract Segal algebras.

\vskip1.0em

Recall that a commutative Banach algebra $\cA$ is said to satisfy
{\it Ditkin's condition} if for every $a\in\cA$ and $x\in\oA$ with $\widehat{a}(x)=0$, there exists a sequence
$(a_n)_{n\in\Bbb{N}} \subseteq \cA$ and open neighbourhoods $(V_n)_{n\in \Bbb{N}}$ of $x$ such that $\widehat{a}_n|_{V_n}\equiv 0$ for all $n\in\Bbb{N}$, and $\lim_n \norm{aa_n-a}_{\cA}=0$.
Furthermore, if $\oA$ is not compact, then for every $a\in\cA$, there must exist a
sequence $(a_n)_{n\in\Bbb{N}}\subseteq \cA_c$ such that $\lim_n \norm{aa_n-a}_\cA =0$. \cite[Proposition~1]{fo0} implies that every semisimple regular commutative Banach algebra $\cA$ which satisfies Ditkin's condition has a $1$-bounded approximate identity.

\begin{proposition}\label{p:Ditkin-and-M-Tauberian}
Let $\cA$ be a regular semisimple Banach algebra satisfying Ditkin's condition. Then $\cA$ is a Tauberian algebra satisfying  condition $(M)$.
\end{proposition}

\begin{proof}
%Clearly, $\cA_c$ is dense in $\cA$.
Given $\epsilon>0$ and  $a\in\cA$ with $\widehat{a}(x)=0$ for some $x\in \oA$. By Ditkin's condition, there is some $b\in\cA$ and  a neighborhood $V$ of $x$ such that $b|_V\equiv 0$ and $\norm{a-ba}_\cA < \epsilon$. Without loss of generality, suppose that $V$ is relatively compact. By Lemma~\ref{l:1-bdd-approximate-identity}, since $\cA$ has a $1$-bounded approximate identity, there is some $a_{V} \in \cA_c$ such that $a_V|_V\equiv 1$ and $\norm{a_{V}}_\cA < 1 +\epsilon$. Let us define $c_V:=a_{V} - a_Vb\in \cA_c$. Hence, $\widehat{c}_V|_V \equiv 1$ and $\norm{c_Va}_\cA \leq \norm{a_{V}}_\cA \norm{a-ba}_\cA < (1+\epsilon)\epsilon$.
\end{proof}

In the rest of this section, we will see more examples of Banach algebras satisfying condition $(M)$.

% \texttt{Prove it! Is the other side correct! $A_p(G)$ is an $M$-Tauberian algebra. Go for Tauberian's
%Segal algebra W(R) (see the classic book of Reiter), or, say Mirkil's algebra
%(a favourite example in the above-mentioned book of Kaniuth}
%%%%%%%%%%%%%%%%%%%%%%%%%%%%%%%%%%%%%%%%%%%%%%%%%%%%%%%%%%%%%%%%%%%%%%%%%%%%%%%%%%%%%%%%%%%%%%%

\begin{subsection}{Lipschitz algebras}\label{ss:Lip-algebra}
Let $\Lpa$ denote the {\it Lipschitz algebra} on a metric
%compact infinite
space $(X,d)$ for some $0<\alpha\leq 1$; see \cite{sh}.
For each $f\in \Lpa$, the Lipschitz norm of $f$ is defined as
\[
\norm{f}_{\Lpa}:=\sup_{x\in X} |f(x)| + \sup_{ x,y\in X, x\neq y} \frac{|f(y)- f(x)|}{d(y,x)^\alpha}.
\]
Also, $\lp$ denotes the subalgebra of $\Lpa$ which consists of all $f$ such that
\[
\frac{ |f(x)-f(y)|}{d(x,y)^\alpha} \rightarrow 0\ \text{as $d(x,y)\rightarrow 0$}.
\]
%%Also, $\lp0$ is the subalgebra of $\lp$ consisting those elements which vanishes at infinity.
Throughout this subsection the metric space $(X,d)$ will be assumed complete. It is convenient and there is no loss of generality in doing so, \cite{sh}.
Let $f$ be a real-valued function defined on the metric space $(X,d)$ and let $k>0$. The {\it truncation} $Tkf$ of $f$  is a function on $X$ defined by
\[
T_kf(x):=\left\{
\begin{array}{l c}
k  &  \text{for all $x\in X$ such that $k \leq f(x)$,}\\
f(x)& \text{for all $x\in X$ such that $-k\leq f(x)\leq k$,}\\
-k & \text{for all $x\in X$ such that $k\geq f(x)$.}\\
\end{array}\right.
\]
If $f$ is a function on $X$ such that $|f(x)-f(y)|\leq K d(x,y)^\alpha$ for some $K\geq 0$, $T_kf\in \lp$.

If $\cA$ is either of the algebras $\Lpa$ and $\lp$, as two commutative Banach algebras, $X$ is dense in $\sigma(\cA)$ in the Gelfand topology, and
the relative Gelfand topology of $X$ coincides with the $d$-topology of $X$. Furthermore, $\cA$ is always  a regular commutative Banach algebra.

%Since for a compact metric space $X$, the natural evaluation map $\Phi: X \rightarrow \Omega_{\Lpa}$ is a homeomorphism onto $\Omega_\Lpa$,  as a dense subset in $\Omega_\cA$,
%we may identify each element of $\Omega_{\Lpa}$ with its correspondence in $\Omega_\Lpa$.
Recall that a metric space $X$ is {\it uniformly discrete} if there is a uniform lower bound on the distance between any two points of $X$.
The following proposition characterizes the condition $(M)$ for Lipschitz algebras. 
Further, it shows that there are Banach algebras which are not satisfying condition $(M)$.

\begin{proposition}\label{p:Lip-algebras}
Let   $X$ be a metric space and let $\cA$ be either of the banach algebras $\Lpa$ or $\lp$. Then
$(i)\Rightarrow (ii) \Rightarrow (iii)$, where
%$\cA$ satisfies  condition $(M)$ if and only if $X$ is finite.
%the followings are equivalent.
\begin{itemize}
\item[$(i)$]{$X$ is uniformly discrete,}
\item[$(ii)$]{$\cA$ satisfies condition $M$,}
\item[$(iii)$]{$X$ is discrete.}
%\item[$(iii)$]{$\cA$ is amenable.}
\end{itemize}
\end{proposition}

\begin{proof}
$(i)\Rightarrow (ii)$.\\
Let $X$ be uniformly discrete for some constant $d_X$ i.e. $d(x,y)>d_X$ for all $x,y\in X$. Also let $\phi_0 \in \Omega_\cA$ such that for some $f\in \cA$, $\phi_0(f)=0$.  
For each given $\epsilon>0$, there is an open neighborhood $U_\epsilon \subseteq \Omega_\cA$ of $\phi_0$ such that such that $|\phi(f)|<\epsilon$ for all $\phi\in U_\epsilon$. 
%Since $X$ is dense in $\Omega_\cA$ and $f$ is continuous on $\Omega_\cA$,

Let $g_\epsilon$ be the characteristic function on $U_\epsilon \cap X$ i.e. $g_\epsilon(x)=1$ if  $x\in U_\epsilon$ and $g_\epsilon(x)=0$ on $X\setminus U_\epsilon$. Since $X$ is uniformly discrete, $g_\epsilon \in \cA$ where $\norm{g_\epsilon}_\cA\leq 1 + d_X^{-\alpha}$.

Furthermore, on one hand $\norm{fg_\epsilon}_\infty <\epsilon$. On the other hand,
\[
\frac{|f(x)g_\epsilon(x)-f(y)g_\epsilon(y)|}{d(x,y)^\alpha} \leq \frac{2\epsilon}{d_X^\alpha}\ \ \ \ \text{for all $x,y\in X$.}
\]
Hence, $\norm{fg_\epsilon}_\cA<\epsilon(2+d_X^{-\alpha})$. To show that $\phi(g_\epsilon)=1$ for all $\phi\in U_\epsilon$, one should note that $X\cap U_\epsilon$ is dense in $U_\epsilon$.

 $(ii)\Rightarrow(iii)$. Suppose that $X$  has a non-isolated point $x_0\in X$. Given $g\in\cA$ such tha $g(x)=d(x,x_0)^\alpha$ for all $x\in X$. Therefore, for truncation of $g$, $T_1g(x_0)=0$. So for each $f\in \cA$ with $f(x_0)=1$, one gets
\[
\norm{f T_1g}_{\cA} \geq \frac{|f(x)T_1g(x)- f(x_0)T_1g(x_0)|}{d(x,x_0)^\alpha}=|f(x)|
\]
for every $x\in \{y\in X: d(x_0,y)<1\}$. Hence, $
\norm{f T_1g}_{\cA} \geq \sup_{x\in X\setminus \{x_0\}} |f(x)| \geq 1
$. It follows that $\cA$ does not satisfy condition $(M)$, which is a contradiction.
\vskip1.5em

\end{proof}

\begin{rem}
(i) Note that for a metric space $(X,d)$, condition $(M)$ of $\Lpa$  and $\lp$ is implied by the amenability, \cite[Theorem~3.1]{dns}.
\end{rem}

(ii) Let $x_0$ be a non-isolated point of $X$. For given $x\in X\setminus \{x_0\}$, and any $0<D$ such that $D^{-1}< d(x,x_0)$, %where $d(X)=\sup_{x,y\in X} d(x,y)$,
 let $U=\{y\in X:\ d(x,y)<D^{-1}\}$.
 %Since $X$ is connected, there exists some $x'\in \overline{U}\setminus U$.
% Hence, $d(x,x_0) > D^{-1}$.
If $f\in \Lpa$ such that $f(x)=1$ and $\supp(f)\subseteq U$, one can verify that $f(x_0)=0$. So,
\[
\lim_{y\in U\setminus\{x\}, y\rightarrow x_0}  \frac{|f(x)-f(y)|}{d(x,y)^\alpha} < D \leq \norm{f}_{\Lpa} -1.
\]
Therefore, $\Lpa$ is not $(D+1)$-uniformly regular for any $D$ greater that $\sup\{d(x,x_0)^{-1}:\ x\in X\}$.

\end{subsection}

%%%%%%%%%%%%%%%%%%%%%%%%%%%%%%%%%%%%%%%%%%%%%%%%%%%%%%%%55
\begin{subsection}{Figa-Talamanca Herz Lebesgue algebras}\label{ss:Figa-talamanca}

Let $G$ be a locally compact group. We denote the Figa-Talamanca Herz algebra by $A_p(G)$, for each $1<p<\infty$. $A_2(G)$ is known as the Fourier algebra of $G$ is defined and studied extensively by Eymard in \cite{ey}. Following \cite{gr}, let us define ${{A_p^r}}(G):=A_p(G)\cap L^r(G)$ for all $1<p<\infty$ and $1\leq r \leq \infty$ equipped with $\norm{\cdot}_{A_p^r(G)}:=\norm{\cdot}_{A_p(G)} + \norm{\cdot}_r$. $A_p^r(G)$ is a semisimple regular commutative Banach algebra, under pointwise multiplication which is called {\it Figa-Talamanca Herz Lebesgue algebra}. Moreover, the maximal ideal space of $A^r_p(G)$
 is $G$, for any, $1 < p < \infty$, $1 \leq r \leq \infty$, \cite[Theorem~1]{gr}. Note that when $r=\infty$, ${{A_p^r}}(G)=A_p(G)$.
Based on \cite[Theorem~1]{gr}, ${{A_p^r}}(G)$ is an abstract Segal algebra with respect to $A_p(G)$ for all $1< p <\infty$ and $1\leq r \leq \infty$. It also has been shown that where $1\leq r < \infty$ and $G$ is not compact, ${{A_p^r}}(G)\neq A_p(G)$ for all $1<p<\infty$. The following lemma is a straightforward result of \cite[Proposition~3.1]{de2}.

\begin{lemma}\label{l:properties-of-Ap}
Let $G$ be a locally compact group and $K$ be a compact subset of $G$, and let $U$ be an open subset of $G$ such that $K \subset U$.
For each $1<p<\infty$ and $V$  a relatively compact open neighborhood of $e$ such that $KVV \subseteq U$, we can find $f_{V}$ in $A_p(G) \cap C_c(G)$ such that {$f_V(G)\subseteq [0,1]$, $f_{V}|_{K} \equiv 1$, $\supp(f_{V}) \subseteq U$, and}
$
\|f_{V}\|_{A_p(G)} \leq (\lambda(KV)/\lambda(V))^{1/p'}$.
\end{lemma}

\begin{cor}\label{c:A_p-is-S-algebra}
Let $G$ be a locally compact amenable group. Then for each $1<p<\infty$, $A_p(G)$ has a $1$-bounded approximate identity.
\end{cor}

\begin{proof}
If $G$ is amenable, it satisfies the {\it Leptin condition} i.e. for every $\epsilon>0$ and compact set $K\subseteq  G$, there
exists a relatively compact neighborhood $V$ of $e$ such that  $\lambda(KV)/\lambda(V) <1+\epsilon$, \cite[Section~2.7]{pi}.
So by the last condition in Lemma~\ref{l:properties-of-Ap}, we may choose $f_V$ such that $\norm{f_V}_{A_p(G)}< 1+\epsilon$; therefore, define $e_{K,\epsilon}:=(1+\epsilon)^{-1}f_V$. Hence, the net $(e_{K,\epsilon})_{K,\epsilon}$ forms a $1$-bounded approximate identity of $A_p(G)$ where $K\rightarrow G$ and $\epsilon \rightarrow 0$.
\end{proof}

\begin{lemma}\label{l:MD-algebra-property-of-Figa-Talamanca}
Let $G$ be a locally compact group and $x\in G$ with a fixed relatively compact neighborhood $U$. If $f$ is a function in $A_p^r(G)$ so that $f(x)=0$, then for each $\epsilon>0$ and $D>1$ there exists a relatively compact neighborhood of $x$ say  $V\subset U$   and $h \in {{A_p^r}}(G)$ such that
\begin{enumerate}
\item[\emph{(i)}]{ $\|h\|_{\infty} \leq 1$, $\| h\|_{A_p(G)} \leq  D$, and $\norm{h}_{A_p^r(G)} \leq 2D$. }
\item[\emph{(ii)}]{  $h|_{V} \equiv 1$ and $\supp(h) \subset U$. }
\item[\emph{(iii)}]{If $p=2$,    $\|h f\|_{A_p^r(G)} < \epsilon$.}
\end{enumerate}
\end{lemma}

\begin{proof}
 Let

 \[
 0< \delta :=\frac{1}{2} \min\left\{ \frac{\epsilon D^{-1/2}}{ \lambda(U)}, \frac{\epsilon }{ \lambda(U)^{1/r}}\right\}.
 \]
Since $f \in A_p^r(G)$, it belongs to $C_{0}(G)$.
For $B_{\delta}:=\{z \in \Bbb{C} : |z|<\delta\}$, $W:= f^{-1}\big( B_{\delta}\big)\cap U $ is a relatively compact open subset of $G$ which contains $x$ and $W \subseteq U$. Let $\lambda$ denote the left Haar measure on $G$. Since $\lambda(x)\leq 1$ and $\lambda$ is  a Radon measure, we may shrink $W$ such that $\lambda(W) \leq D^r$.
Moreover, there exists $O$,  a relatively compact neighborhood of $e$, such that $xOO \subseteq W$.  Since $\lambda$ is a Radon measure, we can find $V$, a relatively compact neighborhood of $x$, such that $\lambda(xO) \leq \lambda(VO) \leq D^{p'/2} \lambda(xO) $;
 meanwhile, $VOO \subseteq W$.
 By  Lemma~\ref{l:properties-of-Ap},there exists some $k \in {{A_p^r}}(G)$, such that $k\big(G\big) \subseteq [0,1]$,  $k|_{V}\equiv 1$, $\supp(k) \subseteq W$, and $\|k\|_{A_p(G)} \leq D^{1/2}$.
Define $h:=k^2$ which belongs to ${{A_p^r}}(G)$.
Also, since $\norm{h}_r \leq \lambda(W)^{1/r} < D$; therefore, $\norm{h}_{A_p^r(G)} \leq 2D$ while $\norm{h}_{A_p(G)}\leq D$.
 Let us consider $kf$ which is a function in $C_{c}(G)$ such that ${\supp}(k f) \subseteq W$, $\| k f\|_{\infty} \leq \delta$.

If $p=2$, applying the duality between $C^*(G)$ and $B(G)$ as well as correspondence of the norms of $A_2(G)$ and $B(G)$,  \cite{ey}, one can get
$\|k^2f\|_{A_2(G)} \leq  \|k\|_{A_2(G)}\; \|k f\|_{1}$.
 Hence, since  $\supp(k f) \subset W$,
 \begin{equation}\label{eq:norm-A_p-and-L1}
 \|h f\|_{A_2(G)}  \leq  \|k\|_{A_2(G)}\; \|k   f\|_{1} \leq D^{1/2}\lambda(U) \; \|k   f\|_{\infty} < D^{1/2}\lambda(U) \; \delta < \epsilon/2.
 \end{equation}
Also, $\norm{h  f}_{r} <   \delta \lambda(U)^{1/r}< \epsilon/2$. Hence, $\norm{h  f}_{A_p^r(G)} <\epsilon$ for $p=2$.
\end{proof}

\begin{cor}\label{c:Figa-is-M-algebra}
Let $G$ be a locally compact group.
For each  $1\leq r \leq \infty$ and $1< p <\infty$, $A_p^r(G)$ is a  $1$-uniformly regular Tauberian algebra. Moreover, $A_2^r(G)$ satisfies condition $(M)$ for all $1\leq r \leq \infty$. Specifically,  $L^1(G)$ is  $1$-uniformly regular commutative Banach  algebra which satisfies condition $(M)$ for  any locally compact abelian group $G$.
\end{cor}

Note that for a non-amenable locally compact group $G$, $A_2(G)$ is a $1$-uniformly regular Banach algebra which does not have any $1$-bounded approximate identity; therefore, it does not satisfy Ditkin's condition.
\end{subsection}

%%%%%%%%%%%%%%%%%%%%%%%%%%%%%%%%%%%%%%%%%%%%%%%%%%%%%%%%%%%%%
\begin{subsection}{Mirkil algebra}\label{ss:Mirkil-alegbra}

 In this subsection we briefly introduce {\it Mirkil algebra}. All the properties and results are presented from \cite[Section~5.4]{ka}.
 In the following we identity the torus, $\Bbb{T}$, with the interval $[-\pi,\pi]$. The commutative Banach algebra
 \[
 \cM:=\left\{ f\in L^2(\Bbb{T}):\; f|_{[-\pi/2, \pi/2]} \ \text{is continuous} \right\}
\]
equipped with norm $\norm{f}_\cM:= \sqrt{2\pi}\norm{f}_2+ \norm{f|_{[-\pi/2, \pi/2]}}_\infty$ and the convolution
\[
f*g(x):=\int_{\Bbb{T}} f(x-t) g(t) dt
\]
is called Mirkil algebra.  For each $n\in \Bbb{Z}$, define $e_n(t):=e^{int}$, $t\in [-\pi,\pi]$; then, the linear space generated by $(e_n)_{n\in\Bbb{Z}}$ is a dense subset of $\cM$.  One may show that $\Omega_\cM$ is homeomorphic to $\Bbb{Z}$; moreover, $\cM$ is a semisimple regular commutative  Banach algebra.  Also, for each $E\subseteq \Bbb{Z}$, $J(E):=\{f\in\cM:\; \widehat{f}(n)=0 \ \text{for all $n\in E$}\}$ is equal to the linear span of $\{e_{n}: \; \notin  E\}$.

\begin{proposition}\label{p:properties-of-Mirkil}
The Mirkil algebra, $\cM$, is a $1$-uniformly regular Tauberian algebra which satisfies condition $(M)$.
\end{proposition}

\begin{proof}
First of all, note that for each $n\in\Bbb{Z}$, $\widehat{e}_{n}(n)=1$; while, $\widehat{e}_m(n)=0$ for all $m\neq n$.
Therefore, $\widehat{e}_n=\delta_n$. Hence, by Lemma~\ref{l:discrete-SP}, $\cM$ is a $1$-uniformly regular Tauberian algebra which satisfies condition $(M)$.
\end{proof}

\begin{rem}\label{r:Mirkil-is-not-Ditkin}
Note that $\cM$ does not satisfy  Ditkin's condition while does  condition $(M)$, \cite[Theorem~5.4.17]{ka}.
\end{rem}

\begin{rem}\label{r:Mirkil-is-asa-of-L2(G)}
One may verify that $\cM$ is an abstract Segal algebra of $L^2(\Bbb{T})$. Therefore, by Proposition~\ref{p:inheritent}, $L^2(\Bbb{T})$, as a commutative Banach algebra, satisfies condition $(M)$.
 \end{rem}

\end{subsection}

\begin{subsection}{Commutative algebras on hypergroups}\label{ss:hypergroup-examples}
Suppose that $H$ is a locally compact hypergroup. The Fourier space of $H$, $A(H)$, is a Banach space defined in \cite{mu}.
For each discrete regular Fourier hypergroup $H$ i.e. $A(H)$ is a Banach algebra, $A(H)$ forms a semisimple regular commutative algebra whose maximal ideal space is $H$. Moreover, $A(H)$ satisfies some properties similar to the ones of Lemma~\ref{l:properties-of-Ap}, \cite[Lemma~3.4]{ma}.  Using the Micheal topology in the definition of hypergroups, \cite{bl}, an argument, similar to the one in the proof of Lemma~\ref{l:MD-algebra-property-of-Figa-Talamanca}, implies that $A(H)$ actually is a $D$-uniformly regular Tauberian algebra satisfies condition $(M)$ for $D=\lambda(e_H)$ where $\lambda$ is the left Haar measure on $H$ and $e_H$ is the identity of the hypergroup $H$.
%From now on, without loss of generality, we suppose that $\lambda(e_H)\leq 1$ for a hypergroup $H$.

\begin{lemma}\label{l:ZL(G)}
Let $G$ be a compact group. Then, $ZL^1(G)$ is a  Tauberian algebra which satisfies condition $(M)$ and has a $1$-bounded approximate identity.
 \end{lemma}

\begin{proof}
Let $G$ be a compact group. Then the set of all irreducible unitary representations, $\widehat{G}$, forms a discrete commutative hypergroup. In \cite[Theorem~3.7]{ma}, it has been shown that the center of the group algebra of $G$, $ZL^1(G)$, is isometrically isomorphic to the Fourier algebra of hypergroup $\widehat{G}$, $A(\widehat{G})$. The $1$-bounded approximate identity is a straightforward result of this fact that  $G$ is a SIN group.
\end{proof}

\end{subsection}

\end{section}

%%%%%%%%%%%%%%%%%%%%%%%%%%%%%%%%%%%%%%%%%%%%%%%%%%%%%%%%%%%%%%%%%%%%%%%%%%%%%%%%%%%%%%%
%%%%%%%%%%%%%%%%%%%%%%%%%%%%%%%%%%%%%%%%%%%%%%%%%%%%%%%%%%%%%%%%%%%%%%%%%%%%%%%%%%%%%%%
\vskip2.0em
%%%%%%%%%%%%%%%%%%%%%%%%%%%%%%%%%%%%%%%%%%%%%%%%%%%%%%%%%%%%%%%%%%%%%%%%%%%%%%%%%%%%%%%
%%%%%%%%%%%%%%%%%%%%%%%%%%%%%%%%%%%%%%%%%%%%%%%%%%%%%%%%%%%%%%%%%%%%%%%%%%%%%%%%%%%%%%%
\begin{section}{ Automatic continuity of separating maps on Tauberian algebras satisfying condition $(M)$}\label{s:automatic-continuity}

\begin{dfn}\label{d:separating-maps}
Let $\cA$ and $\cB$ be two Banach algebras. The linear map
$T:\cA \rightarrow \cB$ is said to be {\it separating} or
{\it disjointness preserving} if $a_1\cdot a_2\equiv0$ implies that $Ta_1 \cdot Ta_2 \equiv 0$ for all $a_1,a_2\in\cA$.\\
\end{dfn}

From now on, let $\cA$ and $\cB$ be two commutative  Banach algebras and $\oA$ and $\oB$ denote the maximal ideal spaces of $\cA$ and $\cB$, respectively.
 We may
define ${\ooA}$ to be the one point compactification of the locally compact space $\oA$. It is clear for each $a\in\cA$, $\widehat{a}$,
as an element of $C_{0}(\oA)$, has a unique extension into $C(\ooA)$.
%Notice that, considering Definition~\ref{d:definitnig-mapping-t}, we may present the following proposition form \cite[Proposition~3]{fo}.
Let $\oB'$ be the set of all ${y} \in \oB$ for which there exists some $a \in \cA$ such that
$\widehat{Ta}({y})\neq0$. Therefore,
\[
\oB' = \bigcup \{ (\widehat{Ta})^{-1}(\Bbb{C}\backslash\{0\}) : a \in \cA_c\}
\]
This implies that $\oB'$ is an open subset of $\oB$. For each ${y} \in \oB'$, we define $\Tg:\cA \rightarrow \Bbb{C}$ as $\Tg(a)=\widehat{Ta}({y})$ for $a \in \cA$.
An open subset $V$ of $\ooA$ is called a  {\it vanishing set} for $\Tg$ if $\Tg(a)=0$ for all $a\in \cA$ that $\coz( \widehat{a})\subseteq V$.

\begin{proposition}\label{p:support-t-is-singelton}\cite[Lemma~1]{fo}\\
Let $T$ be a separating  map from $\cA$ into $\cB$ for semisimple regular commutative  Banach algebras $\cA$ and $\cB$.
Then the  set
\[\supp(\Tg):=\ooA \setminus \cup\{V \subset \ooA : V \
\textrm{is a vanishing set for $\Tg$}\}\]
 is a singleton for each $y \in \oB'$.
\end{proposition}

\begin{dfn}\label{d:definitnig-mapping-t}
Let $T$ be a  separating  map from $\cA$ into $\cB$ for semisimple regular commutative  Banach algebras $\cA$ and $\cB$.
Proposition \ref{p:support-t-is-singelton} lets us define $t: \oB' \rightarrow \ooA$ where for each ${y} \in \oB'$,
$t({y})$ is the solitary element of $\supp({\Tg})$. We call $t$ the {\it support map} of $T$.
\end{dfn}

\begin{proposition}\label{p:properties-of-t}\cite[Proposition~3]{fo}\\
Let $T$ be a separating  map from $\cA$ into $\cB$ for semisimple regular commutative  Banach algebras $\cA$ and $\cB$.
Suppose that $U$ is an open subset of $\oA$ and $a \in \cA$. Then the following statements are held
for $t$ the support map of $T$.\\
\begin{enumerate}
\item{The support map $t$ of $T$ is continuous.}
\item{If $\widehat{a}|_{U} \equiv 0$, then $\widehat{Ta} |_{t^{-1}(U)} \equiv 0$.}
\item{ $t(\coz(\widehat{Ta}) \cap \oB') \subseteq \supp(\widehat{a})$.}
\item{If $T$ is injective, then $t(\oB')$ is dense in $\ooA$.}
\end{enumerate}

\end{proposition}

%\begin{rem}
%About the last condition in Proposition~\ref{p:properties-of-t}, we may note that since our definition of vanishing sets is different, we actually have that $t(\oB')$ is dense in $\oA$ as well as $\ooA$. So we may use a more stronger version of \cite[Proposition~3]{fo} to have the last condition.
%\end{rem}

The following proposition is an adoption of \cite[Proposition~4]{fo} for Tauberian algebras.

\begin{proposition}\label{p:continuity-of-Tg}
Let $T$ be a separating  map from a regular Tauberian algebra $\cA$ into a semisimple regular commutative  Banach algebra $\cB$.
If $y\in\oB'$, then the continuity of $\Tg: \cA \rightarrow \Bbb{C}$ implies that $t(y)\in\oA$.
\end{proposition}

\begin{proof}
If for some $y\in \oB'$, $t(y)=\infty$ the infinity point in   $\ooA\setminus \oA$ and $\Tg$ is continuous, we will get a contradiction:  Suppose that $a\in\cA$ such that $\Tg(a)\neq 0$. Since $\cA_c$ is dense in $\cA$, we may approximate $a$ by a sequence $(a_n)_{n\in\Nat} \subseteq \cA_c$.
But for each $n$, ${a_n}$ is constantly $0$ on ${{\ooA \setminus \supp(\widehat{a}_n)}}$ and clearly, $y\in t^{-1}({\ooA \setminus \supp(\widehat{a}_n)})$. Therefore, by Proposition~\ref{p:properties-of-t}, $\Tg(a_n)=0$ for all $n\in\Nat$. Hence, by continuity of $\Tg$, $\Tg(a)=0$ which is a contradiction.
\end{proof}

%\begin{theorem}\label{t:T-is-onto}
%Let $\cA$ be an M-algebra. Then if $T$ is onto, $t:\oB' \rightarrow \oA$ has a continuous extension $t^*: \oB \rightarrow \ooA$.
%\end{theorem}

%The proof is exactly similar to the proof of \cite[Proposition~4]{fh2}, applying the properties of $M$-Tauberian algebras.

\vskip2.0em

Suppose that  $T$ is a separating  map from $\cA$ into $\cB$ for semisimple regular commutative  Banach algebras $\cA$ and $\cB$.
 Let us define the continuous map $X:t^{-1}(\oA)\rightarrow \Bbb{C}$ as follows: Given ${y} \in t^{-1}(\oA)$, let $U$
be a relatively compact neighborhood of $t({y})$ and let $a_U$
be a function in $\cA_c$ such that $\widehat{a}_U|_{U}\equiv 1$.
Then we define
\begin{equation}\label{eq:X-function}
X({y})=\widehat{Ta_{U}}({y}).
\end{equation}
To show that $X$ is well-defined, let  $V$ be another relatively compact neighborhood of $t({y})$ and  take $a_{V}$ similar to $a_{U}$. By Proposition~\ref{p:properties-of-t} and since $\widehat{a}_{U}  -  \widehat{a}_{V}\equiv 0$ on $V \cap U$, we have
$\widehat{Ta}_{U}  -  \widehat{Ta}_{V} \equiv 0$ on $t^{-1}(V \cap U)$. Since we have chosen $U$ and $V$ as the neighborhoods of $t({y})$, we have
${y} \in t^{-1}(V \cap U)$, and it  shows that $\widehat{Ta}_{U}({y})  =  \widehat{Ta}_{V}({y})$.

 To check the continuity of $X$, let $({y}_{\alpha})$ be a net in $t^{-1}(\oA)$ that converges to some ${y} \in t^{-1}(\oA)$; in addition,
let $U$ be a relatively compact neighborhood of $t({y})$. By Proposition~\ref{p:properties-of-t}, $t$ is a continuous function, so
$t^{-1}(U)$ is an open neighborhood of ${y}$. But there is $\alpha_{0}$ such that for each $\alpha\succcurlyeq \alpha_{0}$, we have
${y}_{\alpha} \in t^{-1}(U)$, so $X({y}_{\alpha})=X({y})$.\\

 Let us consider $\cA$ as a subspace of $\big(C_{0}(\oA),\|\cdot\|_{\infty}\big)$. We use $\oB''$ to denote  the subset of $\oB'$ consisting of all elements ${y} \in \oB'$ for them ${\Tg}$ is a continuous map on $\cA$, where $\cA$ is equipped with $\norm{\cdot}_\infty$-norm. Let $\oB^0$  denote the complement of $\oB''$ in $\oB'$.\\

%\begin{proposition}\label{p:properties-of-T}\cite[Proposition~5 \& 6]{fo}\\
%Let $T$ be a separating  map from $\cA$ into $\cB$ for semisimple regular %commutative  Banach algebras $\cA$ and $\cB$.  Then
%\begin{enumerate}
%\item{For each ${y} \in \oB'$, ${y} \in \oB''$ if and only if %$\widehat{Ta}({y})= X({y})\cdot \widehat{a}(t({y}))$ for all $a \in %\cA$.}
%\item{$\oB''$ is a closed subset of $\oB'$. }
%\item{ $t(\oB^{0})$ is a subset of limit points of $\oA$.}
%\end{enumerate}
%\end{proposition}

The following proposition is a condition $(M)$ version of \cite[Proposition~6]{fo}.

\begin{proposition}\label{p:finite-intersection}
Let $T:\cA \rightarrow \cB$ be a separating  map for  regular commutative  Banach algebras $\cA$ and $\cB$ where $\cA$ satisfies condition $(M)$.
For every compact subset $K$ of $\oA$, let $K^{o}$ denote the interior of $K$. Then $t(\oB^{0}) \cap K^{o}$ is finite .\\
\end{proposition}

\begin{proof}  Toward a contradiction, suppose that there is a sequence of distinct elements of $\oB^{0}$, say $({y}_{n})_{n\in\Nat}$, such that $(t({y}_{n}))_{n\in\Nat} \subseteq K^{o}$ for some compact set $K \subseteq \oA$. Therefore, we can assume that $(U_{n})_{n\in\Nat}$ is a pairwise disjoint sequence of relatively compact open subsets of $K^o$
such that $t({y}_{n}) \in U_{n}$ for each $n \in \Nat$. For each $n$, ${y}_{n} \in \oB^{0}$, so there
exists $a_{n} \in \cA$ such that $\widehat{Ta}_{n}({y}_{n})
\neq X({y}_{n})\cdot \widehat{a}_{n}(t({y}_{n}))$. Let $b_{n}:=a_{n} -
(\widehat{a}_{n}(t({y}_{n}))) a_{V}$ where $V$ is a
relatively compact neighborhood of $K$ and $a_{V} \in \cA$ such that $a_V|V\equiv 1$.  By our
assumption about $b_{n}$,
 it is obvious that $\widehat{Tb}_{n}({y}_{n}) \neq 0$ but $\widehat{b}_{n}(t({y}_{n}))=0$. Since $T$ is linear, we may assume that  $|\widehat{Tb}_{n}({y}_{n})| >n$ for each $n \in \Nat$.
Since $\cA$ satisfies condition $(M$), for each $n \in \Bbb{N}$, we can find
%$V_{n} \subseteq U_{n}$ a neighborhood of $t({y}_{n})$ and
  $c_{n} \in \cA$ such that $\widehat{c}_n|_{V_n}\equiv 1$ for some neighborhood $V_n$ of $t(y_n)$,
 $\|c_{n}  b_{n}\|_{\cA} < {1 \over 2^{n}}$, and $\supp(\widehat{c}_n)\subseteq U_n$. Now, we define $d:=\sum_{n} d_n$ which belongs to $\cA$ where $d_n:=c_{n}  b_n$ for each $n\in\Nat$.
Based on our assumption about $(U_{n})$  and  since $\supp(\widehat{d}_{n}) \subset U_{n}$,
we have $\widehat{d}_{n}|_{U_{m}}\equiv 0$ for each $m \neq n$. By Proposition~\ref{p:properties-of-t}, since $(b_n- d_n)|_{V_n}\equiv 0$ and $(d- d_n)|_{U_n}\equiv 0$, $|\widehat{Td}({y}_{n})|=|\widehat{Td}_{n}({y}_{n})| = |\widehat{Tb}_{n}({y}_{n})|>n$ for each $n$, and it leads to the unboundedness of $Td$ which is a contradiction.
\end{proof}

So we have all necessary tools to develop the main theorem of \cite{fo} for Tauberian algebras satisfying condition $(M)$. The proof, can be written step by step, using the results that we have adapted for condition $(M)$ above, following the proof of \cite[Theorem~1]{fo}. So we omit the proof here.

\begin{proposition}\label{p:automatic-continuity-MAIN-theorem}
Let $T$ be a bijective separating  map from a  Tauberian algebra $\cA$ that satisfies condition $(M)$ onto a semisimple regular commutative Banach algebra $\cB$. Then
$\oB''=\oB$ and $T$ is continuous.
Moreover, $T^{-1}$ is a separating map.

If $\cB$ is also a regular Tauberian algebra satisfying condition $(M)$; then, $t$ is a homeomorphism from $\oB$ onto $\oA$.
Furthermore, $X$ is a non-vanishing contentious map on $\oB$ and $\widehat{Ta}({y})= X({y})\cdot \widehat{a}(t({y}))$ for all $a \in \cA$ and ${y}\in \oB$.
\end{proposition}

Since for lots of examples which we introduced, the maximal ideal space of the algebras satisfying condition $(M)$ are locally compact groups,
note that in Proposition~\ref{p:automatic-continuity-MAIN-theorem} the homeomorphism $t$ does not respect the group structures of these examples necessarily. In Appendix~\ref{s:algebraic-characterization-of-groups}, we see that $t$ may become an isomorphism in the category of topological groups if one imposes some extra conditions.

In the following we see some properties of the mapping $X$, where $\cA$ is a $D$-uniformly regular commutative Banach algebra for some $D>0$.

\begin{proposition}\label{p:above-boundedness-of-X}
Let $T$ be a separating  map from a $D$-uniformly regular commutative Banach algebra $\cA$,  for some $D>0$, into a regular commutative Banach algebra $\cB$.
 Then the mapping $X$ defined above is  bounded on $\oB''$.
\end{proposition}

\begin{proof}
If $X$ is not bounded there is a sequence $({y}_{n})$ in $\oB''$ such that
$|X({y}_{n})|>4^{n}$ for each $n \in\Bbb{N}$. If $(t({y}_{n}))_{n\in\Nat}$ was a finite set in $\oA$, we can
assume that $t({y}_{n})=x$ for all $n \in \Bbb{N}$. Then, there is some $a \in \cA$ such that
$\widehat{a}(x)=1$. Since ${y}_{n} \in \oB''$, we have
\[
|\widehat{Ta}({y}_{n})|=|X({y}_{n})\cdot \widehat{a}(t({y}_{n}))|=|X({y}_{n})|\cdot |\widehat{a}(x)| > 4^{n}
\]
which is contradictory. So without loss of generality, we suppose that $(t(y_n))_{n\in\Nat}$ is a sequence of distinct elements in $\oB''$. Let $(U_{n})_{n\in\Nat}$ be a sequence of pairwise disjoint open subsets of $\oA$ such that
$t({y}_{n}) \in U_{n}$ for each $n \in \Bbb{N}$. By the definition of $D$-uniformly regular Banach algebras, there exists a sequence
$(a_{n})$ in $\cA$ such that $\supp(\widehat{a}_{n}) \subseteq U_{n}$, $\widehat{a}_{n}(t({y}_{n}))=1$, and $\|a_{n}\|_{\cA} \leq (D+1)$ for each $n$. Define $b_{n}:=a_{n}/2^{n}$.
%, so $\widehat{b}_{n}(t({y}_{n}))=D/2^{n}$.
Subsequently, we may  define $b:=\sum_{n} b_{n}$ which is an element of  $\cA$. Since
$(U_{n})_{n\in\Nat}$ are pairwise disjoint open neighborhoods, $\widehat{b}|_{U_{n}}\equiv \widehat{b}_{n}|_{U_{n}}$ for each $n$. Applying Proposition~\ref{p:properties-of-t}, $
|\widehat{Tb}({y}_{n})|=|\widehat{Tb}_{n}({y}_{n})|=|X({y}_{n})  \widehat{b}_n(t({y}_{n}))| >  2^{n}$,
which is a contradiction since $\widehat{Tb} \in  C_0(\oB)$.
\end{proof}

\begin{proposition}\label{p:below-boundedness-of-X}
Let $T$ be a bijective separating  map from a  Tauberian algebra $\cA$ satisfying condition $(M)$  onto  a ${D}$-uniformly regular commutative Banach algebra $\cB$, for some $D>0$. Then there exists $r>0$ such that $X({y})>r$ for each ${y} \in \oB$.
\end{proposition}

\begin{proof}
Suppose that
$({y}_{n})\subset \oB$ is a distinct sequence such that $|X({y}_{n})|<{1/(D+2)^{n}}$ for each $n \in \Bbb{N}$. One can find a sequence $(U_{n})$ of
pairwise disjoint  open sets in $\oB$ such that $U_{n}$ is a neighborhood of ${y}_{n}$  for each $n\in\Nat$. Since $\cB$ is a ${D}$-uniformly regular algebra,
there exists $b_{n} \in \cB$ such that $\widehat{b}_{n}({y}_{n})=1/(D+1)^{n}$, $\supp(\widehat{b}_{n}) \subseteq U_{n}$, and $\|\widehat{b}_{n}\|_{\cB}\leq (D+1)^{-n+1}$ for each $n \in \Bbb{N}$. Let us define $b:=\sum_{n\in\Bbb{N}} b_{n} \in\cB$. Since $T$  is a bijection, there exists $a \in \cA$ such that $Ta=b$. Hence   $|\widehat{a}(t({y}_{n}))|>((D+2)/(D+1))^{n}$, since $|\widehat{Ta}({y}_{n})|=|\widehat{b}({y}_{n})|=|X({y}_{n})|\cdot |\widehat{a}(t({y}_{n}))|$ for each $n\in \Bbb{N}$. But as we have seen in Proposition~\ref{p:automatic-continuity-MAIN-theorem},
$t$ is injective, so it leads to a contradiction with respect to the boundedness of $\widehat{a}$.
\end{proof}

Therefore, we may summarize the results of this section  about the separating maps between $D$-uniformly regular Tauberian algebras satisfying condition $(M)$ in the following theorem.

\begin{theorem}\label{t:Summarize}
Let $T$ be a bijective separating  map from $\cA$  onto $\cB$ for a ${D_\cA}$-uniformly regular Tauberian algebra $\cA$ satisfying condition $(M)$ and a ${D_\cB}$-uniformly regular commutative Banach algebra $\cB$ satisfying condition $(M)$, for some $D_\cA,D_\cB>0$.  Then
$T$ is continuous, $\widehat{Ta}({y})= X({y})\cdot \widehat{a}(t({y}))$ for all $a \in \cA$ and ${y}\in \oB$, where $t$ is a homeomorphism from $\oB$ onto $\oA$, and $X$  is a  bounded continuous  function  on $\oB$ which is also bounded away from $0$.
\end{theorem}

\vskip2.0em
Let us finish this section with a result about extending bijective separating maps.

\begin{cor}\label{c:extending-to-C_0}
Let $T$ be a bijective separating  map from a ${D_\cA}$-uniformly regular Tauberian algebra $\cA$ satisfying condition $(M)$ onto a ${D_\cB}$-uniformly regular commutative Banach algebra $\cB$ satisfying condition $(M)$, for some $D_\cA,D_\cB>0$. Then $T$ has a unique continuous extension to a bijective separating  map $\oT:C_{0}(\oA)\rightarrow C_{0}(\oB)$.
\end{cor}

\begin{proof}
Let $a\in C_0(\oA)$. Then there exists a sequence $(a_n)_{n\in\Nat}\subseteq \cA$ such that $\widehat{a}_n \rightarrow a$. We claim that $(\widehat{Ta}_n)_{n\in \Bbb{N}}$ is a Cauchy sequence in $C_0(\oB)$. Since $T$ is a bijective separating map, for each $n\in \Nat$ and $y\in \oB$,
$\widehat{Ta}_n(y)=X(y) \widehat{a}_n(t(y))$. Therefore,
\begin{eqnarray*}
\norm{\widehat{Ta}_n - \widehat{Ta}_m}_\infty &=& \sup_{y\in\oB} |X(y) \widehat{a}_n(t(y)) - X(y) \widehat{a}_m(t(y))| \\
&\leq& \norm{X}_\infty \sup_{y\in\oB} |\widehat{a}_n(t(y))-\widehat{a}_m(t(y)))| \leq \norm{X}_\infty \norm{\widehat{a}_n - \widehat{a}_m}_\infty.
\end{eqnarray*}
So let us define $\oT (a) = \lim_n T(a_n)$. Hence, $\oT$ is a continuous linear map. The above argument also implies that $\oT$ is one-to-one.
To see that $\oT$ is onto, we need the boundedness of $X$ from $0$ proven in Proposition~\ref{p:below-boundedness-of-X}. Let $(b_n)_{n\in \Nat} \subseteq \cB$ be a $\norm{\cdot}_\infty$-Cauchy. Therefore, there is a sequence $(a_n)_{n\in\Nat}$ such that $Ta_n=b_n$ for each $n\in\Nat$.
Therefore, if $0<r=\inf_{y\in\oB} |X(y)|$, we get
\begin{eqnarray*}
\norm{\widehat{a}_n- \widehat{a}_m}_\infty = \sup_{y\in\oB} |\widehat{a}_n(t(y))- \widehat{a}_m(t(y))| &\leq& r^{-1} \sup_{y\in\oB} | X(y) \widehat{a}_n(t(y))- X(y) \widehat{a}_m(t(y))| \\
&\leq& r^{-1} \norm{ \widehat{b}_n - \widehat{b}_m}_\infty.
\end{eqnarray*}
Therefore, there is some $a\in C_0(\oA)$ such that $a_n\rightarrow a$ and $\oT a=b$. The uniqueness is a direct result of density of $\cA$ in $C_0(\oA)$.
\end{proof}

\end{section}

\vskip2.0em

%%%%%%%%%%%%%%%%%%%%%%%%%%%%%%%%%%%%%%%%%%%%%%%%%%%%%%%%%%%%%%%%%%%%%%%%%%%%%%%%%%%%%%%
%%%%%%%%%%%%%%%%%%%%%%%%%%%%%%%%%%%%%%%%%%%%%%%%%%%%%%%%%%%%%%%%%%%%%%%%%%%%%%%%%%%%%%%
%%%%%%%%%%%%%%%%%%%%%%%%%%%%%%%%%%%%%%%%%%%%%%%%%%%%%%%%%%%%%%%%%%%%%%%%%%%%%%%%%%%%%%%

\begin{section}{Characterization of  bijective separating  maps}\label{s:BSE-condition}

 In this  section, first we recall a family of commutative algebras from \cite{ta}, BSE-algebras. These algebras are a generalization of Fourier-Stieltjes algebras on amenable locally compact groups. Moreover, we introduce some examples of BSE-algebras. Subsequently,   we study separating maps on  regular Tauberian BSE-algebras satisfying condition $(M)$.

Recall that a multiplier $T:\cA\rightarrow \cA$ is a bounded linear operator which
satisfies $a(Tb) = (Ta)b$ for every $a,b \in \cA$. $M(\cA)$ denotes the commutative
Banach algebra consisting of all multipliers on $\cA$. Moreover, for each $T\in M(\cA)$, there is a bounded continuous function $\phi:\oA \rightarrow \Bbb{C}$ such that $\widehat{Ta}(x)=\phi(x)\widehat{a}(x)$ for all $a\in \cA$ and $x\in \oA$, \cite[Proposition~2.2.16]{ka}.

\vskip2.0em
%%%%%%%%%%%%%%%%%%%%%%%%%%%%%%%%%%%%%%%%%%%%%%%%%%%%%%%%%%%%%%%%%%%%%%%%%%%
%\begin{subsection}{BSE-algebras, definition and examples}\label{ss:BSE-algebras}
\begin{dfn}\label{d:BSE-condtion}
For a commutative Banach algebra $\cB$, a function  $\varphi$ on $\oB$ is said to  satisfy the {\it BSE-condition} if there is  $C>0$ such that for every finite set $\{y_1,\cdots,y_n\}\subset \oB$ and $\{\alpha_1,\cdots,\alpha_n\}\subset \Bbb{C}$, we get
\[
\Big| \sum_{i=1}^n \alpha_i \varphi(y_i)\Big| \leq C \Big\| {\sum_{i=1}^n \alpha_i y_i } \Big\|_{\cB^*}.
\]
An algebra $\cB$ is called a {\it BSE-algebra} if the set of continuous functions on $\oB$ which satisfy the BSE-condition equals to $M(\cB)$.
\end{dfn}

 In \cite{ka2}, some conditions on semisimple commutative Banach algebras were studied which make a Banach algebra a BSE-algebra.

 \begin{eg}\label{eg:BSE-algebras}
Let $G$ be a locally compact amenable group. Then the Fourier algebra of $G$, $A_2(G)$, is a BSE-algebra. The proof is based on \cite[Corollary~2.24]{ey}.
Hence, for amenable group $G$, $A_2(G)$ is a  regular Tauberian BSE-algebra  which has a $1$-bounded approximate identity. Moreover, some ideals of $A_2(G)$ with bounded approximate identities are BSE-algebras, \cite{ka2}. The disk algebra and Hardy algebras are also BSE-algebras, \cite{ta}.
 \end{eg}

 \begin{proposition}\label{eg:BSE-ZL1(G)}
 Let $G$ be a compact group. Then $ZL^1(G)$ is a regular Tauberian BSE-algebra satisfying condition $(M)$ which has a $1$-bounded approximate identity.
 \end{proposition}

\begin{proof}
In \cite{ka2}, it is shown that for every compact commutative hypergroup $H$ whose dual, $\widehat{H}$, is another hypergroup, the hypergroup algebra, $L^1(H)$, is a BSE-algebra.
Let $G$ be a compact group. Then the set of all conjugacy classes of $G$, denoted by $\conj(G)$, forms a compact commutative hypergroup whose dual is the hypergroup $\widehat{G}$, \cite{bl}, since  $ZL^1(G)$ is isometrically isomorphic to $L^1(\conj(G))$. Therefore, $ZL^1(G)$ is a BSE-algebra. And by Lemma~\ref{l:ZL(G)}, the proof is complete.
\end{proof}

%\end{subsection}

\vskip2.0em
%%%%%%%%%%%%%%%%%%%%%%%%%%%%%%%%%%%%%%%%%%%%%%%%%%%%%%%%%%%%%%%%%%%%%%%%%%%
%\begin{subsection}{Characterizing separating maps on BSE-algebras}\label{ss:BSE-characterization}

\begin{theorem}\label{t:BSE-condition-T-decomposition}
Let $\cA$ and $\cB$ be two  Tauberian  BSE-algebras satisfying condition $(M)$ which have $1$-bounded approximate identities.
Then the linear  mapping $T:\cA \rightarrow \cB$ is a bijective separating map if and only if $T=T_{2} \circ T_{1}$ where $T_{1}: \cA \rightarrow \cB$ is an algebra isomorphism and $T_{2}\in M(\cB)$.
\end{theorem}

\begin{proof}
  First suppose that $T$ is a bijective separating map. By the results in Section~\ref{s:automatic-continuity}, $\widehat{T(a)}(y)=X(y)\widehat{a}(t(y))$. We claim that $X\in M(\cB)$. Let $\alpha_{1}, \cdots, \alpha_{n}$ be  constants in $\Bbb{C}$ and
$y_{1}, \cdots, y_{n}$ be elements of $\oB$ such that $\| \sum_{i=1}^{n} \alpha_{i} y_i \|_{\cB^*}\leq 1$.
Consider $K:=\{t(y_{1}), \cdots,t(y_{n})\}$ as a compact set in $\oA$ and  $\epsilon >0$ arbitrary. Since $\cA$ has a $1$-bounded approximate identity, by Lemma~\ref{l:1-bdd-approximate-identity}, for each $ \epsilon>0$, there is some $u_{K}\in \cA$ such that $u_{K}|_{K}\equiv 1$ and $\norm{u_{K}}_{\cA} < (1+\epsilon)$.
 If we denote $Tu_{K}$ by $b$, we can write
\begin{eqnarray*}
\bigg|\sum_{i=1}^{n} \alpha_{i} X(y_{i})\bigg| &=& \bigg| \sum_{i=1}^{n} \alpha_{i} X(\gamma_{i}) \cdot \widehat{u}_{K}(t(y_{i})) \bigg| = \bigg| \sum_{i=1}^{n} \alpha_{i} \widehat{b}(y_{i}) \bigg|\\
%&=&  \bigg| \sum_{i=1}^{n} \alpha_{i}\langle \delta_{\gamma_{i}}, g \rangle \bigg|
&\leq&  \|b\|_{\cB}  \|\sum_{i=1}^{n} \alpha_{i} {y_{i}}\|_{\cB^*}
\leq \|Tu_{K}\|_{\cB}
\leq \|T\|  (1+\epsilon).
\end{eqnarray*}
Therefore, $\bigg|\sum_{i=1}^{n} \alpha_{i} X(y_{i})\bigg|  \leq \|T\|$.
Hence, $X \in M(\cB)$ for the BSE-algebra $\cB$.

Now, we will prove that the mapping $a\rightarrow b$ where $\widehat{b}=\widehat{a} \circ t$ for all $a \in \cA$ is an algebra isomorphism. Let $y \in \oB$ and $b_y \in \cB$ such that $b_y(y)=1$. According to  Proposition~\ref{p:automatic-continuity-MAIN-theorem}, $S:=T^{-1}$ is a bijective separating  map, so
$\widehat{Sb}(x)=Y(x)\widehat{b}(s(x))$ for all $b\in\cB$ and $x \in \oA$ when $Y:\oA \rightarrow \Bbb{C}$ is a continuous map defined  similar to $X$ and the support map $s$ of $S$ which is the inverse of $t$, the support map of $T$.  Now we have
\[
1 =  \widehat{b}_y(y) = \widehat{TSb}_{y}(y)= X(y) \cdot \widehat{Sb}_{y}(t(y))=  X(y) Y(t(y))  \widehat{b}_y(s(t(y)))= X(y)  Y(t(y))
\]
which  shows that $X(y) Y(t(y)) =1 $ for each $y\in\oB$.

By the first part of the proof, $Y \in M(\cA)$ and  $Y \cdot (\widehat{b}\circ s)= \widehat{Sb} \in \widehat{\cA}$. Since $\cA$ is an ideal in $M(\cA)$,  $Y \cdot Y \cdot (\widehat{b} \circ s) \in \widehat{\cA}$. Now we can consider
\[
T(Y \cdot Y \cdot (\widehat{b} \circ s) )(y) = X(y) \cdot Y(t(y)) \cdot Y(t(y)) \cdot \widehat{b}(s(t(y)))= Y(t(y)) \cdot \widehat{b}(y)
\]
 for an arbitrary $b\in\cB$ and $y\in \oB$. This implies that $(Y \circ t) \cdot \widehat{b}$ is a function in $\widehat{\cB}$ for all $b \in \cB$. Since $\cB$ is a $BSE$-algebra, $Y \circ t$ is a function in $M(\cB)$. Eventually, since $\cB$ is an ideal in $M(\cB)$ and $X \cdot (\widehat{a} \circ t) \in \cB$ for all $a\in\cA$, $(Y \circ t) \cdot X \cdot (\widehat{a} \circ t) = \widehat{a} \circ t$ belongs to $\cB$.

Let us define $T_{1}: \cA \rightarrow \cB$ where $T_{1}a = b$ for $b\in\cB$ such that $\widehat{b}= \widehat{a} \circ t$ for each $a \in \cA$
and $T_{2}:\cB \rightarrow \cB$ where $Tb=X \cdot b$ for each $b \in \cB$.
 We claim that $T_{1}$ is an injective algebra homomorphism. Let us suppose that $T_1$ is  not onto, so there exists $b \in \cB$ such that $T_{1}a \neq b$ for each $a \in \cA$. Since $T_{2}$ is defined on $\oB$, we can write $Ta = T_{2}(T_{1}a) \neq T_{2}b$ for all $a \in\cA$ which is impossible.

The converse is trivial.
\end{proof}

The bijective separating map $T$ in Theorem~\ref{t:BSE-condition-T-decomposition} is called
 a {\it weighted isomorphism}, see \cite[Section 3.1]{al}.

\begin{cor}\label{c:BSE-condition-isomorphicness}
Let $\cA$ and $\cB$ be two Tauberian   BSE-algebras satisfying condition $(M)$ which have $1$-bounded approximate identities.
Then the existence of a bijective separating map $T:\cA \rightarrow \cB$ implies that $\cA$ and $\cB$ are algebraically isomorphic.
\end{cor}

% \begin{proof}
%The proof is basically based on Theorem~\ref{t:BSE-condition-T-decomposition} where $\ker(T_2)=\{0\}$. But the later fact is resulted from Proposition~\ref{p:below-boundedness-of-X}.
%\end{proof}
%\end{subsection}
\end{section}

%%%%%%%%%%%%%%%%%%%%%%%%%%%%%%%%%%%%%%%%%%%%%%%%%%%%%%%%%%%%%%%%%%%%%%%%%%%%%%%%%%%%%%%%%%%%%
%\vskip2.5em
%%%%%%%%%%%%%%%%%%%%%%%%%%%%%%%%%%%%%%%%%%%%%%%%%%%%%%%%%%%%%%%%%%%%%%%%%%%%%%%%%%%%%%%%%%%%%
%\begin{section}{Further questions}\label{s:further-questions}

%  \begin{itemize}
 % \item{The question is that if we still can reduce some of these conditions in Definition~\ref{d:M-algebra}?}
% \item{Do we know any semisimple regular commutative Banach algebra which is not an $M$-Tauberian algebra? A brief study by the authors suggests that Lipschitz algebras do not seem to be $M$-Tauberian algebras. But it is still a conjecture.}
%  \item{Do we know any regular commutative Banach algebra which is not  $D$-uniformly regular for any $D>0$?}
%\item{One may show that for each $1<p<\infty$ and $1\leq r \leq \infty$, $A_p^r(G)$  satisfies condition $(M)$ if an %equation similar to (\ref{eq:norm-A_p-and-L1}) can be proven for $p$.}
%  \end{itemize}
% \end{section}
%%%%%%%%%%%%%%%%%%%%%%%%%%%%%%%%%%%%%%%%%%%%%%%%%%%%%%%%%%%%%%%%%%%%%%%%%%%%%%%%%%%%%%%%%%%%%
\vskip2.5em
%%%%%%%%%%%%%%%%%%%%%%%%%%%%%%%%%%%%%%%%%%%%%%%%%%%%%%%%%%%%%%%%%%%%%%%

\noindent {\bf Acknowledgement.} The first author was supported by a Dean's Ph.D. scholarship at the University of Saskatchewan for  part of this research. The first author also would like to express his  gratitude to Yemon Choi and Ebrahim Samei, his supervisors, for their essential help to this project and their constant encouragement.
%%%%%%%%%%%%%%%%%%%%%%%%%%%%%%%%%%%%%%%%%%%%%%%%%%%%%%%%%%%%%%%%%%%%%%%%%%%%%%%%%%%%%%%%%%%%%
\vskip2.5em
%%%%%%%%%%%%%%%%%%%%%%%%%%%%%%%%%%%%%%%%%%%%%%%%%%%%%%%%%%%%%%%%%%%%%%%

\appendix
\begin{section}{ Appendix: Algebraic characterization of locally compact groups}\label{s:algebraic-characterization-of-groups}

In this appendix, we briefly study a specific class of regular Tauberian algebras whose maximal ideal spaces are locally compact groups. The existence of a bijective separating map between this type of Tauberian algebras leads to a group isomorphism between maximal ideal spaces. In the following we first define this family of algebras and then we observe some examples of them, and eventually, we study the bijective separating maps between them.  Let $\gA{}$ be a commutative algebra whose maximal ideal space equipped with the Gelfand topology is homeomophic with the locally compact group $G$. Therefore, the elements of $\gA{}$ can be considered as continuous functions on $G$ and their algebraic action can be interpreted as pointwise multiplication on $G$. For each $f\in \gA{}$ let $\widehat{f}$ denote the Gelfand transform of $f$.

\begin{dfn}\label{d:M*-algebras}
For a locally compact group $G$, we call a (semisimple) regular Tauberian  algebra $\gA{}$, a {\it convolution function algebra} over the group $G$ if
\begin{enumerate}
\item[(i)]{ the maximal ideal space of $\gA{}$ is the locally compact group $G$, }
\item[(ii)]{for  all $f,g \in \gA{}$, there is some $h\in\gA{}$ such that $\widehat{h}=\widehat{f}*\widehat{g}$ where $*$ denotes the convolution of group algebra $G$, $L^1(G)$,}
\item[(iii)]{  for all $f\in \gA{}$ and $x\in G$, there exists some $h\in\gA{}$ such that $\widehat{h}=L_x\widehat{f}$ where $L_x\widehat{g}(y):=\widehat{g}(x^{-1}y)$ for all $x,y\in G$ and $g\in C_c(G)$.}
\end{enumerate}
 \end{dfn}

% Therefore, each $M^*$-Tauberian algebra $\gA{}$  represents two algebras, a regular semisimple commutative one by the multiplication of the algebra, and one whose algebra product is inherited from the convolution of the group algebra, $L^1(G)$.
In this section, we denote $\widehat{f}$ by $f$ and $f*g$ denotes  $h\in \gA{}$ where $\widehat{h}=\widehat{f}*\widehat{g}$ for all $f,g\in\gA{}_c$ for a convolution function algebra $\gA{}$.

\begin{eg}\label{eg:C_0(X)-is-M*}
Let $G$ be a locally compact group. Then ${\cal LC}_0(G):=C_0(G)\cap L^1(G)$ is a convolution function algebra over the group $G$. Clearly, $\sigma({\cal LC}_0(G))=G$ as an abstract Segal algebra of $C_0(G)$ meanwhile it is a Segal algebra on $G$ as well. Therefore, ${\cal LC}_0(G)$ is a convolution function algebra on $G$.
\end{eg}

 \begin{eg}\label{eg:A21-is-M*}
$A_2^1(G)$ equipped with pointwise multiplication is a convolution function algebra
% $1\leq r \leq \infty$ where $G$ is compact and for
on every locally compact group $G$. Note that $A_2^1(G)$ represents two Banach algebras, one with pointwise product and one with group algebra convolution which is called {\it Lebesgue Fourier algebra}. Moreover, $A_2^1(G)$ forms a Segal algebra of $G$ as well as an abstract Segal algebra of $A(G)$, \cite[Proposition~2.2]{gl}.
\end{eg}

% \begin{eg}\label{eg:Mirkil-is-M*-algebra}
%The Mirkil  algebra, $\cM$, is an $M^*$-Tauberian algebra over the discrete group $\Bbb{Z}$.
%%To see this fact, just note that $\cM_c$ is the  linear span of $(e_n)_{n\in \Bbb{Z}}$ and for each pair $n,m\in \Bbb{Z}$, %$\widehat{e}_n*\widehat{e}_m=\widehat{e}_{n+m}$.
%Note that for each $a,b\in \cM \subseteq L^2(\Bbb{T})$, ${\cal F}^{-1}(\widehat{a}*\widehat{b})= ab$ which is clearly belongs to %$\cM$, where ${\cal F}^{-1}$ is the inverse of the Fourier transform.    Moreover, as we saw, $\widehat{e}_n=\delta_n$ for each %$n\in \Bbb{Z}$; hence, $L_n\widehat{a}=\widehat{ae}_n$.
%  \end{eg}

For  convolution function algebras, the results of Section~\ref{s:automatic-continuity} can be promoted under a specified condition, $(P)$, that is defined as follows.

\begin{dfn}\label{d:(P)-property}
Let $\gA{1}$ and $\gA{2}$ be two convolution function algebras over locally compact groups $G_1,G_2$, respectively.
A linear operator $T:\gA{1} \rightarrow \gA{2}$ satisfies  condition $(P)$ if for all $f,g \in \gA{1}$ and
${y} \in G_{2}$ such that $T(f * g)({y})=0$, we have $Tf * Tg({y})=0$.
\end{dfn}

The following theorem is the main result that we may prove using  condition $(P)$.

\begin{theorem}\label{t:main-theorem-(P)-property}
Let $\gA{1}$ and $\gA{2}$ be two  convolution function  algebras over locally compact groups $G_1,G_2$ satisfying condition $(M)$.
If $T$  is a bijective separating map from $\gA{1}$ onto $\gA{2}$ such that satisfies  condition $(P)$, its support map, $t$,
is a topological group isomorphism from $G_{2}$ onto $G_{1}$, i.e. it is a topological homeomorphism that meanwhile acts as a group isomorphism.
\end{theorem}

The proof is a direct result of
the following lemma and Proposition~\ref{p:automatic-continuity-MAIN-theorem}. The proof of the following lemma, also, is exactly similar to the proof of \cite[Lemma~2]{fh3}, so we omit the proof here.

\begin{lemma}\label{l:(P)-property-lemma}
Let $\gA{1}$ and $\gA{2}$ be two convolution function algebras on locally compact groups $G_1,G_2$, respectively.
Suppose that $T:\gA{1} \rightarrow \gA{2}$ is a map such that  $Tf=X\cdot f\circ t$ for each $f\in \gA{1}$ where $X$ is a non vanishing scalar valued continuous function on $G_{2}$ and $t$
is a homeomorphism map from $G_{2}$ onto $G_1$. If $T$ satisfies condition $(P)$, then $t$ is a group homomorphism.
\end{lemma}

\end{section}

%\footnotesize

 \vspace{5mm}

 \noindent
   M. Alaghmandan\\
   Department of Mathematics and Statistics,\\
     University of Saskatchewan,\\
     Saskatoon, SK, S7N 5E6, Canada\\
       {\bf Email:} mahmood.a@usask.ca\\

    \noindent
      R. Nasr-Isfahani and M. Nemati\\
   Department of Mathematical Sciences,\\
     Isfahan Uinversity of Technology,\\
      Isfahan 84156-83111, Iran\\
       {\bf Emails:}  isfahani@cc.iut.ac.ir, m.nemati@math.iut.ac.ir

       \end{document}